\documentclass[12pt, reqno]{amsart}
\usepackage{amsmath}
\usepackage{amsfonts}
\usepackage{amssymb}
\usepackage{version}
\usepackage[all]{xy}
\begin{document}
\numberwithin{equation}{section}

\def\1#1{\overline{#1}}
\def\2#1{\widetilde{#1}}
\def\3#1{\widehat{#1}}
\def\4#1{\mathbb{#1}}
\def\5#1{\frak{#1}}
\def\6#1{{\mathcal{#1}}}

\newcommand{\de}{\partial}
\newcommand{\R}{\mathbb R}
\newcommand{\Ha}{\mathbb H}
\newcommand{\al}{\alpha}
\newcommand{\tr}{\widetilde{\rho}}
\newcommand{\tz}{\widetilde{\zeta}}
\newcommand{\tk}{\widetilde{C}}
\newcommand{\tv}{\widetilde{\varphi}}
\newcommand{\hv}{\hat{\varphi}}
\newcommand{\tu}{\tilde{u}}
\newcommand{\tF}{\tilde{F}}
\newcommand{\debar}{\overline{\de}}
\newcommand{\Z}{\mathbb Z}
\newcommand{\C}{\mathbb C}
\newcommand{\Po}{\mathbb P}
\newcommand{\zbar}{\overline{z}}
\newcommand{\G}{\mathcal{G}}
\newcommand{\So}{\mathcal{S}}
\newcommand{\Ko}{\mathcal{K}}
\newcommand{\U}{\mathcal{U}}
\newcommand{\B}{\mathbb B}
\newcommand{\oB}{\overline{\mathbb B}}
\newcommand{\Cur}{\mathcal D}
\newcommand{\Dis}{\mathcal Dis}
\newcommand{\Levi}{\mathcal L}
\newcommand{\SP}{\mathcal SP}
\newcommand{\Sp}{\mathcal Q}
\newcommand{\A}{\mathcal O^{k+\alpha}(\overline{\mathbb D},\C^n)}
\newcommand{\CA}{\mathcal C^{k+\alpha}(\de{\mathbb D},\C^n)}
\newcommand{\Ma}{\mathcal M}
\newcommand{\Ac}{\mathcal O^{k+\alpha}(\overline{\mathbb D},\C^{n}\times\C^{n-1})}
\newcommand{\Acc}{\mathcal O^{k-1+\alpha}(\overline{\mathbb D},\C)}
\newcommand{\Acr}{\mathcal O^{k+\alpha}(\overline{\mathbb D},\R^{n})}
\newcommand{\Co}{\mathcal C}
\newcommand{\Aut}{{\sf Aut}(\mathbb D)}
\newcommand{\D}{\mathbb D}
\newcommand{\oD}{\overline{\mathbb D}}
\newcommand{\oX}{\overline{X}}
\newcommand{\loc}{L^1_{\rm{loc}}}
\newcommand{\la}{\langle}
\newcommand{\ra}{\rangle}
\newcommand{\thh}{\tilde{h}}
\newcommand{\N}{\mathbb N}
\newcommand{\kd}{\kappa_D}
\newcommand{\Hr}{\mathbb H}
\newcommand{\ps}{{\sf Psh}}
\newcommand{\Hess}{{\sf Hess}}
\newcommand{\subh}{{\sf subh}}
\newcommand{\harm}{{\sf harm}}
\newcommand{\ph}{{\sf Ph}}
\newcommand{\tl}{\tilde{\lambda}}
\newcommand{\gdot}{\stackrel{\cdot}{g}}
\newcommand{\gddot}{\stackrel{\cdot\cdot}{g}}
\newcommand{\fdot}{\stackrel{\cdot}{f}}
\newcommand{\fddot}{\stackrel{\cdot\cdot}{f}}

\def\v{\varphi}
\def\Re{{\rm Re}\,}
\def\Im{{\rm Im}\,}
\def\ext{{\sf ext}\,}
\def\Hol{{\rm Hol}}
\def\Lr{{\sf Lr}\,}
\def\tr{{\sf tr}\,}
\def\sp{{\rm sp}}
\def\rg{{\sf rg}\,}

\def\Label#1{\label{#1}}

% Standard sets

\def\cn{{\C^n}}
\def\cnn{{\C^{n'}}}
\def\ocn{\2{\C^n}}
\def\ocnn{\2{\C^{n'}}}
\def\je{{\6J}}
\def\jep{{\6J}_{p,p'}}
\def\th{\tilde{h}}

% Abbreviations

\def\dist{{\rm dist}}
\def\const{{\rm const}}
\def\Span{{\rm Span\,}}
\def\rk{{\rm rank\,}}
\def\dim{{\rm dim}}
\def\id{{\rm id}}
\def\aut{{\sf aut}}
\def\Aut{{\sf Aut}}
\def\CR{{\rm CR}}
\def\GL{{\sf GL}}
\def\U{{\sf U}}
\def\la{\langle}
\def\ra{\rangle}

\emergencystretch15pt \frenchspacing

\newtheorem{theorem}{Theorem}[section]
\newtheorem*{theorem**}{Theorem \mynumber}
\newenvironment{theorem*}[1]
  {\newcommand{\mynumber}{\ref{#1}}\begin{theorem**}}
  {\end{theorem**}}
\newtheorem{lemma}[theorem]{Lemma}
\newtheorem{proposition}[theorem]{Proposition}
\newtheorem{problem}[theorem]{Problem}
\newtheorem{corollary}[theorem]{Corollary}
\newtheorem{conjecture}[theorem]{Conjecture}

\theoremstyle{definition}
\newtheorem{definition}[theorem]{Definition}
\newtheorem{example}[theorem]{Example}

\theoremstyle{remark}
\newtheorem{remark}[theorem]{Remark}
\numberwithin{equation}{section}

\title[Basins of attraction in Loewner equations]{Basins of attraction in Loewner equations}

\author[L. Arosio]{Leandro Arosio$^{\ddagger}$}
\address{Istituto Nazionale di Alta Matematica ``Francesco Severi'', Citt\`a Universitaria, Piazzale Aldo Moro 5, 00185 Rome, Italy}
\email{arosio@altamatematica.it}
\thanks{$^{\ddagger}$Titolare di una Borsa della Fondazione Roma - Terzo Settore  bandita dall'Istituto Nazionale di Alta Matematica}

\date{\today }

\begin{abstract}
We prove that any Loewner PDE whose driving term $h(z,t)$ vanishes at the origin, and satisfies the bunching condition $\ell m(Dh(0,t))\geq k(Dh(0,t))$ for some $\ell\in \mathbb{R}^+$, admits a  solution  given by univalent mappings $(f_t\colon \B^q\to\C^q)_{t\geq 0}$. This is done by  discretizing time and considering the abstract basin of attraction. If $\ell<2$, then the range $\cup_{t\geq 0} f_t(\B^q)$ of any such solution is biholomorphic to $\C^q$.

\end{abstract}
\maketitle
%\tableofcontents
\section{Introduction}
Let $\B^q\subset \C^q$ denote the unit ball.
The Loewner PDE 
\begin{equation}\label{PDE} 
\frac{\partial f_t(z)}{\partial t}=Df_t(z)h(z,t),\quad \mbox{a.e.}\ t\geq 0,\ z\in \B^q
\end{equation} 
was introduced by Loewner \cite{Loewner}  and developed by Kufarev \cite{Kufarev} and Pommerenke \cite{Pommerenke} in  the case of the unit disc $\mathbb{D}\doteq \B^1$.
The study of this equation culminated with the proof of the Bieberbach conjecture by de Branges \cite{deBranges} and the introduction of the stochastic Loewner evolution by Schramm \cite{Schramm}.

 The several variables case has been widely studied for its application in geometric function theory by Graham, Hamada,  G.Kohr, M.Kohr, Pfaltzgraff and others (see e.g. \cite{Graham-Kohr} \cite{Pfaltzgraff}).

In \cite{Arosio-Bracci-Hamada-Kohr}, generalizing the results obtained in the unit disc $\mathbb{D}$  in \cite{Contreras-Diaz-Gumenyuk}, we explore the connections between  this topic and the theory recently developed by Bracci, Contreras and D\'iaz-Madrigal  \cite{Bracci-Contreras-Diaz} \cite{Bracci-Contreras-Diaz-II} (see also \cite{Arosio-Bracci}) of Herglotz non-autonomous vector fields on complete hyperbolic manifolds. 
An {\sl Herglotz vector field of order $\infty$} on $\B^q$ is  a non-autonomous holomorphic vector field $-h(z,t)\colon \B^q\times\R^+\to\C^q$ such that 
\begin{itemize}
\item $-h(z,t)$ is measurable in $t\geq 0$ and for a.e. $\bar t\geq 0$, the holomorphic vector field $-h(z,\bar t)$ is an {\sl infinitesimal generator}, that is
the ``frozen'' Cauchy problem
$$
\begin{cases}
 \overset{\bullet}{z}(s)=-h(z(s),\bar t),\\
z(0)=z_0,
\end{cases}
$$
has a solution $z:[0,+\infty)\to \B^q$ for all $z_0\in \B^q$,
\item for any compact set $K\subset \B^q$ and any $T>0$ there exists  $c_{K,T}>0$ satisfying  $$|h(z,t)|\leq c_{K,T},\quad z\in K, 0\leq t\leq T.$$
\end{itemize}
The solution flow  of the Loewner ODE
\begin{equation}
\begin{cases}
\frac{\de}{\de t} \v_{s,t}(z)=-h(\v_{s,t}(z),t),\quad z\in\B^q,\ \mbox{a.e.}\ t\in [s,\infty),\\
 \v_{s,s}(z)=z ,\quad z\in\B^q, s\geq 0,
\end{cases}
\end{equation}
is an {\sl evolution family of order $\infty$}, that is a family of holomorphic mappings $(\v_{s,t}\colon \B^q\to\B^q)_{0\leq s\leq t}$ satisfying
\begin{itemize}
\item $\v_{s,t}=\v_{u,t}\circ\v_{s,u}$ for all $0\leq s\leq u\leq t$ and $\v_{s,s}(z)=z$ for all $s\geq 0$,
\item  for any compact set $K\subset \B^q$ and for any $T>0$ there exists a  $C_{K,T}>0$ satisfying 
  \begin{equation}\label{ck-evd}
|\v_{s,t}(z)- \v_{s,u}(z)|\leq C_{K,T}(t-u), \quad z\in K,\  0\leq s\leq u\leq t<T.
  \end{equation}
\end{itemize}

In \cite{Arosio-Bracci-Hamada-Kohr} we prove that a family $(f_t\colon \B^q\to\C^q)_{t\geq 0}$ of univalent mappings is locally Lipshitz (in the variable $t$) and solves the Loewner PDE  (\ref{PDE}) if and only if it solves  
 the functional equation
\begin{equation}\label{functional}
f_s=f_t\circ \v_{s,t},\quad 0\leq s\leq t.
\end{equation}
If such a solution   $(f_t\colon \B^q\to\C^q)$ exists,
then the subset $\bigcup_{t\geq 0} f_t(\B^q)\subset \C^q$ is open and connected and is called the {\sl range} of $(f_t)$. 
Any other solution $(g_t\colon \B^q\to \C^q)$ is of the form $(\Lambda\circ f_t)$, where $\Lambda\colon \bigcup_{t\geq 0} f_t(\B^q)\to \C^q$ is holomorphic. Thus the ranges of two univalent solutions of  (\ref{PDE}) are biholomorphic.

We are interested in  Herglotz vector fields on $\B^q$ whose flow $(\v_{s,t})$ is  attracting in the origin. A first example is provided by Herglotz vector fields whose linear part does not depend on $t\geq 0$. This has been studied in \cite{Duren-Graham-Hamada-Kohr}\cite{Graham-Hamada-Kohr-Kohr}. 
\begin{theorem}\label{exkohr}
Let $-h(z,t)$ be a Herglotz vector field of order $\infty$ on $\B^q$ such that $h(z,t)=Az+O(|z|^2)$ with 
\begin{equation}\label{condition}
2\min\{\Re\langle Az,z\rangle:|z|=1\}>\max \{ \Re\lambda :\lambda \in \sp(A)\},
\end{equation} where $\langle\cdot,\cdot\rangle$ is the hermitian product on $\C^q$.
Then the Loewner PDE (\ref{PDE}) admits a locally Lipshitz  univalent solution $(f_t\colon \B^q\to\C^q)$. The range $\bigcup_{t\geq 0} f_t(\B^q)$ of any such solution is biholomorphic to $\C^q$.
\end{theorem}

This result was generalized in  \cite{Arosio-II} (see also \cite{Arosio}), with an approach based on a discretization of time.
\begin{theorem}\label{exmain}
Let $-h(z,t)$ be a Herglotz vector field of order $\infty$ on $\B^q$ such that $h(z,t)=Az+O(|z|^2)$, where the eigenvalues of $A$ have strictly positive real part.
Then the Loewner PDE (\ref{PDE}) admits a locally Lipshitz  univalent solution $(f_t\colon \B^q\to\C^q)$. The range $\bigcup_{t\geq 0} f_t(\B^q)$ of any such solution  is biholomorphic to $\C^q$. 
\end{theorem}
The same result was obtained independently  with  different methods by Voda \cite{Voda}, assuming $\min\{\Re\langle Az,z\rangle:|z|=1\}>0$. See also \cite{Hamada} for related results.

The next natural step is admitting time-dependent linear parts. Set $$m(A)\doteq  \min\{\Re\langle Az,z\rangle:|z|=1\},\quad k(A)=\max\{\Re\langle Az,z\rangle:|z|=1\}.$$ The following result is proved in  \cite{annali}\cite{CAOT}:
\begin{theorem}\label{gabicaot}
Let $-h(z,t)$ be a Herglotz vector field on $\B^q$ of order $\infty$ such that $h(z,t)=A(t)z+O(|z|^2)$, and assume that the family of linear mappings $(A(t))_{t\geq 0}$ satisfies:
\begin{itemize}
\item[i)] $m(A(t))>0$ for all $t\geq 0$ and $\int_0^\infty m(A(t))=\infty$,
\item[ii)] $t\mapsto \|A(t)\|$ is uniformly bounded on $\R^+$,
\item[iii)] there exists $\delta>0$ such that $$2m(A(t))\geq k(A(t))+\delta,\quad t\geq 0,$$
\item[iv)] $$\int_{s}^tA(\tau)d\tau\circ \int_{r}^sA(\tau)d\tau =\int_{r}^sA(\tau)d\tau\circ \int_{s}^tA(\tau)d\tau,\quad t\geq s\geq r\geq 0.$$
\end{itemize}
Then the Loewner PDE (\ref{PDE}) admits a locally Lipshitz univalent solution $(f_t\colon \B^q\to\C^q)$. The range $\bigcup_{t\geq 0} f_t(\B^q)$ of any such solution  is biholomorphic to $\C^q$.
\end{theorem}

In this paper we generalize Theorem \ref{gabicaot}, using the approach of \cite{Arosio}\cite{Arosio-II}. The following is our result.
\begin{theorem}\label{mainintro}
Let $-h(z,t)$ be a Herglotz vector field on $\B^q$ of order $\infty$ such that $h(z,t)=A(t)z+O(|z|^2)$, and assume that the family of linear mappings $(A(t))_{t\geq 0}$ satisfies:
\begin{itemize}
\item[a)] $m(A(t))>0$ for all $t\geq 0$ and $\int_0^\infty m(A(t))=\infty$,
\item[b)] $t\mapsto \|A(t)\|$ is locally bounded on $\R^+$,
\item[c)] there exists $\ell\in \mathbb{R}^+$ such that $$\ell m(A(t))\geq k(A(t)),\quad t\geq 0.$$
\end{itemize}
Then the Loewner PDE (\ref{PDE}) admits a locally Lipshitz univalent solution $(f_t\colon \B^q\to\C^q)$.  If $\ell<2$ then   the range $\bigcup_{t\geq 0} f_t(\B^q)$ of any such solution is biholomorphic to $\C^q$. 
\end{theorem}  
Notice that the assumptions ii) and iii) of Theorem \ref{gabicaot} imply  that there exists $\ell<2$ such that  $$\ell m(A(t))\geq k(A(t)),\quad t\geq 0.$$

We want to stress the  strong analogy between Loewner theory and the theory of  discrete  non-autonomous complex dynamical systems  grown around  Bedford's conjecture (see \cite{Abate-Abbondandolo-Majer}\cite{Fornaess-Stensones}\cite{Jonsson-Varolin}\cite{Peters}\cite{Wold}). This is reflected in the proof of Theorem \ref{mainintro}, which is based on a discretization of time, and  relies on the  study of the abstract basin of attraction performed by Fornaess and Stens\o nes in \cite{Fornaess-Stensones}:
\begin{theorem}\label{squall}
Let $(\v_{n,n+1})_{n\in\N}$ be a family of univalent self-mappings of $r\B^q$. Assume that  there exist $0<\nu\leq\mu<1$ such that 
\begin{equation}\label{peter}
\nu|z|\leq |\v_{n,n+1}(z)|\leq \mu|z|,\quad z\in r\B^q, n\in \N.
\end{equation}
Then, if $\Omega$ is the abstract basin of attraction of $(\v_{n,n+1})$, then there exists an univalent mapping $\Psi\colon \Omega\to\C^q$. 
\end{theorem}
The abstract basin of attraction comes naturally  with a family of univalent mappings $(\omega_n\colon r\B^q\to \Omega)$. Composing this family with the biholomorphism $\Psi$ given by Theorem \ref{squall} we obtain a  family of univalent mappings  $r\B$ to $\C^q$  which we extend  to a family $(f_t\colon \B^q\to \C^q)$ satisfying the functional equation (\ref{functional}).

The range $\bigcup_{t\geq 0} f_t(\B^q)$ is by construction biholomorphic to $\Omega$ and thus by  \cite[Theorem 3.1]{Fornaess-Stensones} it  is  a Stein, Runge domain in $\C^q$ whose  Kobayashi pseudometric vanishes identically and which is diffeomorphic to $\C^q$. It is an open question whether $\bigcup_{t\geq 0} f_t(\B^q)$ is  biholomorphic to $\C^q$ when $\ell\geq 2$.  A positive answer would follow from a proof of  Bedford's  conjecture (see e.g. \cite{Petersthesis}):
\begin{conjecture}
Let $(\Phi_{n,n+1})_{n\in\N}$ be a family of automorphisms of $\C^q$. Assume that  there exist $0<\nu\leq\mu<1$ and $r>0$ such that 
\begin{equation}
\nu|z|\leq |\Phi_{n,n+1}(z)|\leq \mu|z|,\quad z\in r\B^q, n\in \N.
\end{equation}
Then the basin of attraction  $$ \{z\in \C^q: \lim_{n\to\infty}\Phi_{n-1,n}\circ \cdots\circ\Phi_{0,1} = 0\}$$ is biholomorphic to $\C^q$.

\end{conjecture}

\section{Main result}
Let $\mathcal{N}$ denote the family of holomorphic mappings $h\colon \B^q\to \C^q$ such that  $h(0)=0$ and $\Re\langle h(z),z \rangle >0,$ for all $z\neq 0$.

\begin{theorem}\label{main}
Let $h(z,t)\colon \B^q\times \R^+\to \C^q$ be a mapping such that 
 $z\mapsto h(z,t)\in \mathcal{N}$ for all $t\in \R^+$ and $t\mapsto h(z,t)$ is measurable on $\R^+$ for all $z\in \B^q$.
Assume that $h(z,t)=A(t)z+O(|z|^2)$  and that the family of linear mappings $(A(t))_{t\geq 0}$ satisfies:
\begin{itemize}
\item[a)] $m(A(t))>0$ for all $t\geq 0$ and $\int_0^\infty m(A(t))=\infty$,
\item[b)] $t\mapsto \|A(t)\|$ is locally bounded on $\R^+$,
\item[c)]  there exists $\ell\in \mathbb{R}^+$ such that $$\ell m(A(t))\geq k(A(t)),\quad t\geq 0.$$
\end{itemize} 
Then the Loewner PDE
$$\frac{\partial f_t(z)}{\partial t}=Df_t(z)h(z,t),\quad z\in\B^q,\ \mbox{a.e.}\ t\geq 0$$ admits a global solution given by univalent mappings $(f_t\colon \B^q\to \C^q)$. If $l<2$, then the range $\bigcup_{t\geq 0} f_t(\B^q)$ of any such solution is biholomorphic to $\C^q$.
Any other solution given by holomorphic mappings  $(g_t\colon\B^q\to\C^q)$ is of the form $(\Lambda\circ f_t)$, where $\Lambda\colon \bigcup_{t\geq 0}f_t(\B^q)\to \C^q$ is holomorphic. 

\end{theorem}
\begin{proof}
Notice that for all $A\in \mathcal{L}(\C^q)$, $$m(A)\leq k(A)\leq \|A\|,$$ and thus $k(t)$ and $m(t)$ are also locally bounded on $\R^+$.
By \cite[Lemma 1.2]{Graham-Hamada-Kohr-Kohr} one has for a.e. $t\geq 0$ $$|h(z,t)|\leq\frac{4r}{(1-r)^2}\|A(t)\|,\quad |z|\leq r<1,$$ hence $-h(z,t)$ is a Herglotz vector field of order $\infty$ on $\B^q$. Let $(\v_{s,t})$ be the associated evolution family of order $\infty$, that is the solution of the  Loewner ODE 
\begin{equation}
\begin{cases}
\frac{\de}{\de t} \v_{s,t}(z)=-h(\v_{s,t}(z),t),\quad z\in\B^q,\ \mbox{a.e.}\ t\in [s,\infty),\\
 \v_{s,s}(z)=z ,\quad z\in\B^q, s\geq 0.
\end{cases}
\end{equation} 
Recall that $\v_{s,t}\colon\B^q\to\B^q$ is an univalent mapping for all $0\leq s\leq t$ and that
 $t\mapsto \v_{s,t}(z)$ is locally Lipshitz continuous on $[s,\infty)$ uniformly on compact sets with respect to $z \in \B^q$.

Fix $s\geq 0$ and $z\in \B^q\smallsetminus\{0\}$.   
Then for a.e. $\tau\geq 0$, $$\frac{\de}{\de \tau} |\v_{s,\tau}(z)|^2=2\Re\left\langle\frac{\de}{\de\tau} \v_{s,\tau}(z),\v_{s,\tau}(z)\right\rangle=-2\Re\langle h(\v_{s,\tau}(z),\tau),\v_{s,\tau}(z)\rangle.$$
Set $C(r)\doteq\frac{1+r}{1-r}$ and $c(r)\doteq\frac{1-r}{1+r}$ for all $r\geq 0.$ Gurganus proved \cite{Gurganus} that for a.e. $t\geq 0$,
$$\Re\langle A(t)w,w\rangle c(|w|)\leq \Re\langle h(w,t),w\rangle\leq  \Re\langle A(t)w,w\rangle C(|w|),\quad w\in \B^q\smallsetminus\{0\}.$$

Since $|\v_{s,\tau}(z)|\leq |z|$ one has using assumption d), $$-2k(A(\tau))C(|z|)\leq \frac{\frac{\de}{\de \tau} |\v_{s,\tau}(z)|^2}{ |\v_{s,\tau}(z)|^2}\leq -2m(A(\tau)) c(|z|),\quad \mbox{a.e.}\ \tau\geq 0,$$ 
$$-2C(|z|)\int_s^tk(A(\tau))d\tau\leq\int_s^t \frac{\frac{\de}{\de \tau} |\v_{s,\tau}(z)|^2}{ |\v_{s,\tau}(z)|^2}d\tau\leq-2c(|z|)\int_s^tm(A(\tau)) d\tau,\quad 0\leq s\leq t,$$ 
\begin{equation}\label{ultima}
   e^{-C(|z|)\int_s^tk(A(\tau)) d\tau}\leq  \frac{|\v_{s,t}(z)|}{|z|}\leq e^{-c(|z|)\int_s^tm(A(\tau)) d\tau},\quad 0\leq s\leq t. 
\end{equation}
 Set for all $0\leq s\leq t$, $$\nu_{s,t}\doteq e^{-C(|z|)\int_s^tk(A(\tau)) d\tau},\quad\mbox{and}\quad \mu_{s,t}\doteq e^{-c(|z|)\int_s^tm(A(\tau)) d\tau}.$$
One has, thanks to assumption c), 
\begin{equation}\label{aeris}
\log_{\mu_{s,t}} \nu_{s,t}= \frac{\log \nu_{s,t}}{\log \mu_{s,t}}=C^2(|z|)\frac{\int_s^tk(A(\tau))d\tau}{\int_s^tm(A(\tau))d\tau}\leq C^2(|z|)\ell ,\quad 0\leq s\leq t.
\end{equation}
Let $n\in \mathbb{N}$ and let $u_n\in \mathbb{R}^+$ be defined by $$\int_0^{u_n}m(A(\tau))d\tau=n.$$ 
Let now $h\in \mathbb{N}$ be the least integer strictly greater than  $\ell$, and let $r>0$ be such that  $C^2(r)< h/\ell$. Set $\mu\doteq  e^{-c(r)}$ (notice that $\mu=\mu_{u_n,u_{n+1}}$ for all $n\geq 0$) and $\nu\doteq \min\{\nu_{u_n,u_{n+1}}: n\geq 0\}$. By (\ref{ultima}) and  (\ref{aeris})
one has that
$$
\nu|z|\leq |\v_{u_n,u_{n+1}}(z)|\leq \mu|z|,\quad z\in r\B^q, n\geq 0.
$$
and  
\begin{equation}\label{h}
\mu^{h}<\nu.
\end{equation}

The {\sl abstract basin of attraction} or {\sl tail space} $\Omega$ of the family $(\v_{u_n,u_{n+1}}\colon r\B^q\to r\B^q)$ is defined in \cite{Fornaess-Stensones} (see also \cite{Abate-Abbondandolo-Majer}) as its topological inductive limit   endowed with a natural complex structure. $\Omega$ is the quotient of the set $$\left\{z\in \prod_{m\geq n}r\B^q:n\in \N,\ z_{m+1}=\v_{u_m,u_{m+1}}(z_m), \quad 0\leq n\leq m\right\},$$ obtained  identifying $z$ and $z'$ if $z_m=z_m'$ for $m$ large enough, and the holomorphic structure is induced by a family of open inclusions $(\omega_n\colon r\B^q\to \Omega)$ defined as $$\omega_n(z)\doteq (\v_{u_n,u_m}(z))_{m\geq n},\quad n\in\N,$$ which are thus by definition biholomorphisms with their image and  satisfy 
\begin{equation}\label{cloudstrife}
\omega_n(z)=\omega_m \circ \v_{u_n,u_m}(z),\quad 0\leq n\leq m,z\in r\B^q.
\end{equation}

By \cite[Theorem 2.2]{Fornaess-Stensones} there exists an univalent mapping  $\Psi\colon \Omega\to \C^q$. 
We claim that, for all $s\geq 0$, the sequence $(\Psi\circ\omega_m\circ \v_{s,u_m})_{m\geq 0}$ converges  uniformly on compact sets in $\Hol (\B^q,\C^q)$. Indeed by equation (\ref{ultima}) and  assumption a)  one has that for all $s\geq 0$, $$\lim_{m\to\infty} \v_{s,u_m}(z)=0,$$ uniformly on compact sets. Thus, if $0<v<1$, there exist $m(v)\in \mathbb{N}$ such that  for all $j\geq m(v)$, one has  $\v_{s,u_j}(v\B^q)\subset r\B^q.$ Let $j,h$ be integers such that  $m(v)\leq j\leq h$, then by (\ref{cloudstrife}), $$\Psi\circ\omega_h\circ \v_{s,u_h}(z)=  \Psi\circ\   \omega_j \circ\v_{s,u_j}(z),\quad z\in v\B^q.$$ Thus the sequence $(\Psi\circ\omega_m\circ \v_{s,u_m})$ is eventually constant in $\Hol (v\B^q,\C^q)$.

Let $f_t\colon \B^q\to\C^q$ the univalent mapping defined as 
\begin{equation}\label{tidus}
f_t(z)\doteq   \lim_{m\to+\infty} \Psi\circ\omega_m\circ \v_{t,u_m}(z).
\end{equation} 

One easily verifyies that
\begin{equation}\label{aiuto}
f_s(z)=f_t\circ \v_{s,t}(z),\quad 0\leq s\leq t,z\in \B^q,
\end{equation}
and that $$\bigcup_{t\geq 0}f_t(\B^q)=\Psi(\Omega).$$ Notice that the  abstract basin of attraction of the family $(\v_{u_n,u_{n+1}})$ is thus biholomorphic to the Loewner range of the family $(\v_{s,t})$ defined in \cite{Arosio-Bracci-Hamada-Kohr}. This can be  checked directly  since both objects are defined as direct limits.

 By  \cite[Theorem 4.10]{Arosio-Bracci-Hamada-Kohr} one has   that $(f_t\colon \B^q\to \C^q)$ is a Loewner chain of order $\infty$, that is
\begin{itemize}
\item $f_s(\B^q)\subset f_t(\B^q)$ for all $0\leq s\leq t$,
\item  for any compact set $K\subset \B^q$ and for any $T>0$ there exists a  $k_{K,T}>0$ satisfying 
  \begin{equation}\label{ck-evd}
|f_t(z)- f_s(z)|\leq k_{K,T}(t-s), \quad z\in K,\  0\leq s\leq t<T.
  \end{equation}
\end{itemize}
 By  \cite[Theorem 5.2]{Arosio-Bracci-Hamada-Kohr} one obtains  finally $$\frac{\partial f_t(z)}{\partial t}=Df_t(z)h(z,t),\quad z\in \B^q,\ \mbox{a.e.}\ t\geq 0.$$

Thus any univalent mapping  $\Psi\colon \Omega\to \C^q$ gives rise to a solution $(f_t\colon \B^q\to \C^q)$ of the Loewner PDE. Following \cite[Remark A.4]{Abate-Abbondandolo-Majer} we recall a way to construct such  univalent  mapping $\Psi\colon \Omega\to \C^q$. Given any polynomial map $p\colon \C^q\to \C^q$ of degree at most $k$ there exists \cite{Forstneric}\cite{Weickert} an holomorphic  automorphism $\Phi$ of $\C^q$ such that $$\Phi(z)=p(z)+O(|z|^{k+1}).$$ We choose a  sequence of automorphisms  $(\Phi_{n,n+1}\colon \C^q\to\C^q)$ satisfying $$\Phi_{n,n+1}(z)=\v_{u_n,u_{n+1}}(z)+O(|z|^{h}),\quad n\geq 0,$$
where $h\in \mathbb{N}$  is as in (\ref{h}). We denote the basin of attraction of the sequence $(\Phi_{n,n+1})$ by  $$\mathfrak{A}(\Phi_{n,n+1})\doteq \{z\in \C^q: \lim_{n\to\infty}\Phi_{n-1,n}\circ \cdots\circ\Phi_{0,1} = 0\}.$$ It follows from  \cite[Remark A.4]{Abate-Abbondandolo-Majer} that there exists a biholomorphism   $$\Psi\colon \Omega\to \mathfrak{A}(\Phi_{n,n+1}) \subset \C^q.$$

If $\ell<2$, then $h=2$ and, by \cite[Theorem 4]{Wold}, one has that the basin of attraction  $\mathfrak{A}(\Phi_{n,n+1})$ is biholomorphic to $\C^q$.

  By \cite[Theorem 4.10]{Arosio-Bracci-Hamada-Kohr} any solution $(g_t\colon \B^q\to \C^q)$ of the Loewner PDE has to satisfy  $g_s=g_t\circ \v_{s,t}$ for all $0\leq s\leq t$ and thus \cite[Theorem 4.7]{Arosio-Bracci-Hamada-Kohr} yields that the family $(g_t)$ is of the form $(\Lambda\circ f_t)$, where $\Lambda\colon \bigcup_{t\geq 0}f_t(\B^q)\to \C^q$ is holomorphic.

\end{proof}

\end{document}